\documentclass[pdflatex,sn-mathphys-num]{sn-jnl}
\usepackage{graphicx}%
\usepackage{multirow}%
\usepackage{amsmath,amssymb,amsfonts}%
\usepackage{amsthm}%
\usepackage{mathrsfs}%
\usepackage[title]{appendix}%
\usepackage{xcolor}%
\usepackage{textcomp}%
\usepackage{manyfoot}%
\usepackage{booktabs}%
\usepackage{algorithm}%
\usepackage{algorithmicx}%
\usepackage{algpseudocode}%
\usepackage{listings}%
\theoremstyle{thmstyleone}
\newtheorem{theorem}{Theorem}[section]
\newtheorem{proposition}[theorem]{Proposition}
\newtheorem{lemma}[theorem]{Lemma}
\newtheorem{corollary}[theorem]{Corollary}

\theoremstyle{thmstyletwo}

\theoremstyle{thmstylethree}
\newtheorem{definition}[theorem]{Definition}

\begin{document}

\title[Isometries of Unit Spheres of SFS-Spaces and a Geometric Approach to Tingley's Problem]{Isometries of Unit Spheres of SFS-Spaces and a Geometric Approach to Tingley's Problem}

%% Suffix	->  S\mathcal{F}x{IV}

\author[1,2]{\fnm{Karimbergen} \sur{Kudaybergenov}}
\email{kudaybergenovkk@gmail.com}

\author[3]{\fnm{Mukhtar} \sur{Ibragimov}}
\email{m.ibragimov1909@gmail.com}

\author*[3,4]{\fnm{Allabay} \sur{Arziev}}
\email{allabayarziev@gmail.com}

\affil[1]{%
\orgdiv{Institute for Advances Study in Mathematics},
\orgname{Harbin Institute of Technologies},
\orgaddress{
\city{Harbin},
\postcode{150001},
\country{China}}}

\affil[2]{%
\orgname{Suzhou Research Institute of Harbin Institute of Technology},
\orgaddress{
\city{Suzhou},
\postcode{215104},
\country{China}}}

\affil[3]{%
\orgname{V.I. Romanovskiy Institute of Mathematics},
\orgaddress{
\street{Uzbekistan Academy of Sciences},
\city{Tashkent},
\postcode{100174},
\country{Uzbekistan}}}

\affil[4]{%
\orgdiv{Math Department},
\orgname{Karakalpak State University},
\orgaddress{
\city{Nukus},
\postcode{230100},
\country{Uzbekistan}}}


\abstract{
We investigate surjective isometries between the unit spheres of facially symmetric spaces and their interaction with the underlying facial structure. It is shown that such isometries preserve the essential geometric relations associated with symmetric faces, orthogonality, and geometric tripotents. In particular, a surjective sphere isometry induces a natural correspondence between the geometric tripotents of the underlying spaces and preserves the principal structural features determined by the Peirce decomposition. These results provide a geometric framework for studying linear extension problems for sphere isometries.

As an application, we establish a positive solution of Tingley’s problem for finite-dimensional real neutral strongly facially symmetric spaces satisfying the joint Peirce property. We also obtain an extension result for a wider class of real neutral strongly facially symmetric spaces admitting an appropriate approximation by finite-dimensional Peirce subspaces. In these settings, every surjective isometry between unit spheres extends uniquely to a surjective real-linear isometry of the ambient spaces. The results demonstrate that the facial and tripotent structures play a decisive role in the rigidity of spherical isometries.
}




\maketitle

\section{Introduction}
Tingley's problem asks whether every surjective isometry between the unit
spheres of two Banach spaces extends to a surjective real-linear isometry
between the spaces. One approach to this problem is based on studying the
action of a spherical isometry on the faces of the unit ball. The
preservation of maximal convex subsets of the sphere and other geometric
invariants \cite{Tanaka2014,Campos2021} plays an essential role in this
approach, whose effectiveness has been demonstrated for spaces of trace
class operators \cite{Fernandez2017} and preduals of von Neumann algebras
\cite{Mori2018}.

The theory of facially symmetric spaces, developed by Y.~Friedman and
B.~Russo \cite{FR1986,FR1989Affine,FR1989Operator,FR1992,FR1993}, provides
a framework for studying such questions by means of symmetric faces,
geometric tripotents, and Peirce decompositions. The aim of this paper is
to study invariants of surjective isometries between the unit spheres of
these spaces and to apply them to particular cases of Tingley's problem.
We first establish the preservation of orthogonality and of the
intersections of orthogonal complements with the sphere. For neutral
strongly facially symmetric spaces, this yields equality of the ranks of
spaces with isometric unit spheres.

The main part of the study concerns the preservation of symmetric faces.
We prove that every such face of the unit ball of a facially symmetric
space can be represented as the intersection of two faces corresponding
to maximal geometric tripotents
(Proposition~\ref{prop:intersection}). This allows us to establish that a
surjective isometry between the unit spheres of Banach facially symmetric
spaces maps every symmetric face onto a symmetric face
(Theorem~\ref{thm:faces}). Consequently, it induces a unique bijection
between the geometric tripotents that preserves order in both directions
(Theorem~\ref{thm:tripotents}). We then study the preservation of opposite
faces, orthogonality, sums of orthogonal geometric tripotents, and
orthocomplementation. In particular, we obtain oddness of the isometry at
norm-exposed points and, under property $\mathrm{PE}$ (see
Definition~\ref{def:PEFE}), at all extreme points of the unit ball.

The preservation of the unit spheres of Peirce-zero subspaces is
essential for the extension problem. The finite-dimensional case is
treated using the classification of finite-dimensional real neutral
strongly facially symmetric spaces with property $\mathrm{JP}$ (see
Definition~\ref{def:JP}) as $\ell_1$-sums of Hilbert spaces
\cite{Kudaybergenov2020}, together with a result on isometries between the
spheres of $\ell_1$-sums of strictly convex spaces \cite{Wang2002}. These
known results imply the existence of a unique surjective linear
isometric extension for this class
(Corollary~\ref{cor:finiteextension}).

The final result, Theorem~\ref{thm:extension}, concerns real Banach neutral
strongly facially symmetric spaces with property $\mathrm{JP}$. If the domain
space admits a family of finite-dimensional Peirce-zero subspaces that
is directed by inclusion and whose union is norm dense, then every
surjective isometry between the unit spheres extends uniquely to a
surjective linear isometry between the spaces. The proof is based on
the compatibility of local linear extensions and passage to the closure
of the union of the subspaces. Thus, preservation of the facial structure
provides an extension mechanism under the stated approximation
condition.

\section{Preliminaries}

We recall the facts from the theory of facially symmetric spaces needed
to present our results. We use the concepts, terminology, and notation
of \cite{FR1989Operator,FR1992,FR1993,Tanaka2014,Campos2021}.

Let $Z$ be a real or complex normed space. Elements $f,g\in Z$ are said
to be \emph{orthogonal} if
\[
 \left\lVert f+g\right\rVert =\left\lVert f-g\right\rVert =\left\lVert f\right\rVert +\left\lVert g\right\rVert .
\]
In this case we write $f\mathbin{\diamond} g$.

Recall that a convex subset $F$ of the unit ball
\[
 Z_1=\{f\in Z:\left\lVert f\right\rVert \leq 1\}
\]
is called a \emph{face} if, whenever $f,g\in Z_1$, $t\in(0,1)$, and
$tf+(1-t)g\in F$, we have $f,g\in F$. A \emph{norm-exposed face} of
$Z_1$ is a nonempty proper subset of the form
\[
 F_x=\{f\in Z_1:f(x)=1\},
 \qquad x\in Z^*,\quad \left\lVert x\right\rVert =1.
\]
For any subset $S\subset Z$, put
\[
 S^\diamond=\{f\in Z:f\mathbin{\diamond} g\text{ for every }g\in S\},
\]
and call $S^\diamond$ the \emph{orthogonal complement} of $S$. Subsets
$S,T\subset Z$ are said to be orthogonal, written $S\mathbin{\diamond} T$, if
$f\mathbin{\diamond} g$ for every $(f,g)\in S\times T$. An element $u\in Z^*$ is
called a \emph{projective unit} if $\left\lVert u\right\rVert =1$ and
$\langle u,F_u^\diamond\rangle=0$. We denote by $\mathcal{F}$ and
$\mathcal{U}$ the sets of all norm-exposed faces of $Z_1$ and all
projective units in $Z_1^*$, respectively. The unit sphere of $Z$ is
denoted by $\partial Z_1$, that is,
\[
 \partial Z_1=\{f\in Z_1:\left\lVert f\right\rVert =1\}.
\]

A norm-exposed face $F$ of $Z_1$ is called \emph{symmetric} if there
exists a surjective linear isometry $S_F:Z\to Z$ satisfying $S_F^2=I$
(such a map is called a \emph{symmetry}) whose fixed-point space is
\[
 \overline{\operatorname{sp} F}\oplus F^\diamond.
\]

\begin{definition}\label{def:WFS}
A real or complex normed space $Z$ is called a \emph{weakly facially
symmetric space} (a $\mathrm{WFS}$-space) if every norm-exposed face of $Z_1$
is symmetric.
\end{definition}

A \emph{geometric} (or \emph{generalized}) \emph{tripotent} is a
projective unit $u\in Z^*$ such that $F:=F_u$ is a symmetric face and
$S_F^*u=u$ for some symmetry $S_F$ associated with $F$. We denote by
${S\mathcal{F}}$ and ${G\mathcal{U}}$ the sets of all symmetric faces of $Z_1$ and all
geometric tripotents in $Z^*$, respectively.

A partial order on ${G\mathcal{U}}$ is defined by
\[
 u\leq v\quad\Longleftrightarrow\quad F_u\subseteq F_v,
 \qquad u,v\in{G\mathcal{U}}.
\]
In a $\mathrm{WFS}$-space $Z$, each symmetric face $F_u$ determines geometric
(or generalized) Peirce projections $P_k(u)$, $k\in\{0,1,2\}$, as
follows:
\[
 P_1(u)=\frac12(I-S_{F_u}),
 \qquad P_1(u)(Z)=\{f\in Z:S_{F_u}f=-f\}.
\]
The projections $P_0(u)$ and $P_2(u)$ map $Z$ onto $F_u^\diamond$ and
$\overline{\operatorname{sp} F_u}$, respectively. For convenience, we use the notation
\[
 P_k(F_u)=P_k(u),\qquad P_k(u)Z=Z_k(u),\qquad S_{F_u}=S_u,
 \quad k\in\{0,1,2\}.
\]
Clearly,
\begin{gather*}
 P_0(u)+P_2(u)=\frac12(I+S_u),\\
 P_0(u)+P_1(u)+P_2(u)=I,\qquad
 P_0(u)-P_1(u)+P_2(u)=S_u.
\end{gather*}

A geometric tripotent $u\in{G\mathcal{U}}$ is called \emph{minimal} if
$\dim Z_2(u)=1$. We denote by $\mathcal{M}$ the set of all minimal
geometric tripotents.

A geometric tripotent $u\in{G\mathcal{U}}$ is called \emph{maximal} if
$Z_0(u)=\{0\}$. The set of all maximal geometric tripotents on $Z$ is
denoted by ${G\mathcal{U}}_{\max}(Z)$.

\begin{definition}\label{def:dualorth}
Let $Z$ be a normed space. Elements $u,v\in Z^*$ are said to be
\emph{orthogonal} if there exists a symmetric face $F\subset Z_1$ such
that either
\begin{enumerate}
 \item[{\rm(i)}] $u\in\operatorname{im} P_2(F)^*$ and $v\in\operatorname{im} P_0(F)^*$; or
 \item[{\rm(ii)}] $u\in\operatorname{im} P_0(F)^*$ and $v\in\operatorname{im} P_2(F)^*$.
\end{enumerate}
In this case we write $u\mathbin{\diamond} v$ and $v\mathbin{\diamond} u$.
\end{definition}

\begin{definition}\label{def:rank}
A $\mathrm{WFS}$-space $Z$ is said to have \emph{rank $n$}, $n\in\mathbb{N}$, if every
family of pairwise orthogonal geometric tripotents in $Z^*$ has
cardinality at most $n$, and there exists such a family containing
exactly $n$ elements. We write $\operatorname{rank} Z=n$.
\end{definition}

In general,
\[
 \operatorname{rank} Z=\sup\{|G|:G\subset{G\mathcal{U}}\text{ is a pairwise orthogonal family}\}.
\]
Here $|G|$ denotes the cardinality of $G$, and the supremum is understood
in the sense of cardinal numbers.

It was shown in \cite[Lemma~2.5]{FR1986} that if $u,v\in{G\mathcal{U}}$ and
$u\mathbin{\diamond} v$, then $u+v\in{G\mathcal{U}}$.

\begin{definition}\label{def:JP}
A $\mathrm{WFS}$-space $Z$ is said to have property $\mathrm{JP}$ if, for every pair of
orthogonal geometric tripotents $u$ and $v$, we have
\[
 S_uS_v=S_{u+v}.
\]
\end{definition}

\begin{definition}\label{def:SFS}
A $\mathrm{WFS}$-space $Z$ is called a \emph{strongly facially symmetric space}
(an $\mathrm{SFS}$-space) if, for every symmetric face $F_u$ of $Z_1$ and every
$v\in Z^*$ satisfying $\left\lVert v\right\rVert =1$ and $F_u\subset F_v$, we have
$S_u^*v=v$, where $S_u$ is the symmetry associated with $F_u$.
\end{definition}

A contractive projection $Q$ on a normed space $Z$ is called
\emph{neutral} if, for every $f\in Z$, the equality
$\left\lVert Qf\right\rVert =\left\lVert f\right\rVert $ implies $Qf=f$. A normed space $Z$ is called
\emph{neutral} if, for every symmetric face $F_u$, the corresponding
projection $P_2(u)$ is neutral.

A point $x\in Z_1$ is called an \emph{extreme point} if
\[
 x=ty+(1-t)z,\qquad y,z\in Z_1,\quad 0<t<1,
\]
implies $y=z=x$. The set of all extreme points of $Z_1$ is denoted by
$\operatorname{ext} Z_1$. A point $x\in Z_1$ is called a \emph{norm-exposed point}
if there exists $u\in Z^*$ with $\left\lVert u\right\rVert =1$ such that
\[
 \{x\}=\{z\in Z_1:u(z)=1\}.
\]
We denote by $\exp Z_1$ the set of all norm-exposed points of $Z_1$.
In general, $\exp Z_1\subset\operatorname{ext} Z_1$.

\begin{definition}\label{def:PEFE}
An $\mathrm{SFS}$-space $Z$ is said to satisfy axiom
\begin{itemize}
 \item $\mathrm{PE}$ if $\exp Z_1=\operatorname{ext} Z_1$;
 \item $\mathrm{FE}$ if every nonempty norm-closed proper face of the unit ball
 $Z_1$ is norm exposed.
\end{itemize}
\end{definition}

\section{Main results}

\begin{definition}\label{def:maxconvex}
A nonempty subset $C\subset\partial X_1$ is called a \emph{maximal
convex subset} if the following two conditions hold:
\begin{enumerate}[label=\textup{(\arabic*)}]
 \item[{\rm(1)}] $C$ is convex;
 \item[{\rm(2)}] $C$ is maximal by inclusion among convex subsets of
 $\partial X_1$: if $D$ is convex and
 $C\subseteq D\subseteq\partial X_1$, then $C=D$.
\end{enumerate}
\end{definition}

\begin{proposition}\label{prop:orthogonality}
Let $X$ and $Y$ be real or complex normed spaces, and let
$S:\partial X_1\to\partial Y_1$ be a surjective isometry. Then:
\begin{enumerate}[label=\textup{(\roman*)}]
 \item[{\rm (i)}] for every $f,g\in\partial X_1$, we have
 \[
  f\mathbin{\diamond} g\quad\Longleftrightarrow\quad S(f)\mathbin{\diamond} S(g);
 \]
 \item[{\rm (ii)}] for every $A\subseteq\partial X_1$, we have
 \[
  S(A^\diamond\cap\partial X_1)=S(A)^\diamond\cap\partial Y_1;
 \]
 \item[{\rm (iii)}] if $f,g\in\partial X_1$ and $f\mathbin{\diamond} g$, then $f$ and $g$ belong
 to a common maximal convex subset of $\partial X_1$.
\end{enumerate}
\end{proposition}

\begin{proof}
\textup{(i)} Since $S$ is an isometry,
\[
 \left\lVert f-g\right\rVert =\left\lVert S(f)-S(g)\right\rVert .
\]
If $f\mathbin{\diamond} g$ (or $S(f)\mathbin{\diamond} S(g)$), then
\[
 \left\lVert f+g\right\rVert =\left\lVert f\right\rVert +\left\lVert g\right\rVert =2
 \quad\text{(or }\left\lVert S(f)+S(g)\right\rVert =\left\lVert S(f)\right\rVert +\left\lVert S(g)\right\rVert =2\text{)}.
\]
Thus, by \cite[Lemma~2.1]{Fernandez2017},
\[
 \left\lVert f+g\right\rVert =\left\lVert S(f)+S(g)\right\rVert .
\]
Consequently,
\begin{align*}
 f\mathbin{\diamond} g
 &\Longleftrightarrow
 \left\lVert f\right\rVert +\left\lVert g\right\rVert =\left\lVert f+g\right\rVert =\left\lVert f-g\right\rVert \\
 &\Longleftrightarrow
 \left\lVert S(f)\right\rVert +\left\lVert S(g)\right\rVert
 =\left\lVert S(f)+S(g)\right\rVert =\left\lVert S(f)-S(g)\right\rVert ,
\end{align*}
that is, $f\mathbin{\diamond} g$ if and only if $S(f)\mathbin{\diamond} S(g)$.

\textup{(ii)} For $x\in\partial X_1$, we have
\begin{equation}\label{eq:orthcomplement}
\begin{aligned}
 x\in A^\diamond\cap\partial X_1
 &\Longleftrightarrow x\mathbin{\diamond} a\text{ for every }a\in A\\
 &\Longleftrightarrow S(x)\mathbin{\diamond} S(a)\text{ for every }a\in A\\
 &\Longleftrightarrow S(x)\in S(A)^\diamond\cap\partial Y_1.
\end{aligned}
\end{equation}
Hence
\[
 S(A^\diamond\cap\partial X_1)
 \subseteq S(A)^\diamond\cap\partial Y_1.
\]
Conversely, let $y\in S(A)^\diamond\cap\partial Y_1$. Since $S$ is
surjective, there exists $x\in\partial X_1$ such that $S(x)=y$. Clearly,
$y\mathbin{\diamond} S(a)$ for every $a\in A$, so $S(x)\mathbin{\diamond} S(a)$. By
\eqref{eq:orthcomplement}, $x\mathbin{\diamond} a$ for every $a\in A$, and therefore
$x\in A^\diamond\cap\partial X_1$. Thus
\[
 y=S(x)\in S(A^\diamond\cap\partial X_1),
\]
which gives the reverse inclusion.

\textup{(iii)} Since $f\mathbin{\diamond} g$ and $\left\lVert f\right\rVert =\left\lVert g\right\rVert =1$, we have
$\left\lVert f+g\right\rVert =2$. The Hahn--Banach theorem provides a norming functional
$v\in X^*$ with $\left\lVert v\right\rVert =1$ such that $v(f+g)=2$. Since
$\left\lVert v\right\rVert =\left\lVert f\right\rVert =\left\lVert g\right\rVert =1$, we have $|v(f)|\leq 1$ and
$|v(g)|\leq 1$. On the other hand,
\[
 v(f+g)=v(f)+v(g)=2,
\]
and hence $\operatorname{Re}v(f)+\operatorname{Re}v(g)=2$. Since
\[
 \operatorname{Re}v(f)\leq|v(f)|\leq 1,
 \qquad \operatorname{Re}v(g)\leq|v(g)|\leq 1,
\]
we obtain $\operatorname{Re}v(f)=\operatorname{Re}v(g)=1$, and therefore
$v(f)=v(g)=1$. Thus $f,g\in F_v$. By Zorn's lemma, $F_v$ is contained
in a maximal convex subset $C$ of $\partial X_1$, so $f,g\in C$.
\end{proof}

In a neutral $\mathrm{SFS}$-space $Z$, for $f\neq 0$ we denote by $v(f)$ the
unique geometric tripotent (see \cite[Theorem~4.3]{FR1989Affine}) such
that
\[
 \langle f,v(f)\rangle=\left\lVert f\right\rVert ,
 \qquad \langle v(f),\{f\}^\diamond\rangle=0.
\]
By \cite[Corollary~1.3(b), Lemmas~2.1 and~2.5]{FR1986}, for all
$f,g\in Z\setminus\{0\}$ we have
\begin{equation}\label{eq:supportorth}
 f\mathbin{\diamond} g\quad\Longleftrightarrow\quad v(f)\mathbin{\diamond} v(g)
 \quad\Longleftrightarrow\quad F_{v(f)}\mathbin{\diamond} F_{v(g)}.
\end{equation}

The following assertions follow from
Proposition~\ref{prop:orthogonality}.

\begin{corollary}\label{cor:rank}
Let $X$ and $Y$ be neutral $\mathrm{SFS}$-spaces, and let
$S:\partial X_1\to\partial Y_1$ be a surjective isometry. Then
\[
 \operatorname{rank} X=\operatorname{rank} Y.
\]
\end{corollary}

\begin{proof}
Let $\{u_i\}_{i\in I}\subset{G\mathcal{U}}(X)$ be any family of pairwise
orthogonal geometric tripotents. For each $i\in I$, choose
$f_i\in F_{u_i}$. Since $u_i\mathbin{\diamond} u_j$ for $i\neq j$,
\cite[Lemma~2.5]{FR1986} gives $F_{u_i}\mathbin{\diamond} F_{u_j}$, and hence
$f_i\mathbin{\diamond} f_j$. By Proposition~\ref{prop:orthogonality},
$S(f_i)\mathbin{\diamond} S(f_j)$. Put $w_i=v(S(f_i))$. Then
\eqref{eq:supportorth} yields
\[
 w_i\mathbin{\diamond} w_j,\qquad i\neq j.
\]
Thus $\{w_i\}_{i\in I}\subset{G\mathcal{U}}(Y)$ is a family of pairwise
orthogonal geometric tripotents. Every pairwise orthogonal family in
${G\mathcal{U}}(X)$ therefore gives rise to such a family in ${G\mathcal{U}}(Y)$ of the
same cardinality. Hence
\[
 \operatorname{rank} X\leq\operatorname{rank} Y.
\]
The isometry $S^{-1}:\partial Y_1\to\partial X_1$ also preserves
orthogonality, so the same argument gives $\operatorname{rank} X\geq\operatorname{rank} Y$.
Therefore $\operatorname{rank} X=\operatorname{rank} Y$.
\end{proof}

\begin{corollary}\label{cor:orthspan}
Let $X$ and $Y$ be $\mathrm{WFS}$-spaces over the same field $K$, and let
$S:\partial X_1\to\partial Y_1$ be a surjective isometry. Suppose that
$f_1,f_2,\ldots,f_n\in\partial X_1$ are pairwise orthogonal. Then the
map
\[
 T\left(\sum_{i=1}^n\alpha_i f_i\right)
 =\sum_{i=1}^n\alpha_i S(f_i)
\]
is a linear isometry
\[
 T:\operatorname{sp}_K\{f_1,f_2,\ldots,f_n\}
 \longrightarrow\operatorname{sp}_K\{S(f_1),S(f_2),\ldots,S(f_n)\}.
\]
\end{corollary}

\begin{proof}
By \cite[Corollary~1.2]{FR1989Affine} and
Proposition~\ref{prop:orthogonality}\textup{(i)}, we have
\[
 \left\lVert \sum_{i=1}^n\alpha_i f_i\right\rVert =\sum_{i=1}^n|\alpha_i|.
\]
In particular, $f_1,f_2,\ldots,f_n$ are linearly independent, and
therefore
\[
 T\left(\sum_{i=1}^n\alpha_i f_i\right)
 =\sum_{i=1}^n\alpha_i S(f_i)
\]
defines a linear map. Moreover,
\[
 \left\lVert \sum_{i=1}^n\alpha_i f_i\right\rVert
 =\sum_{i=1}^n|\alpha_i|
 =\left\lVert \sum_{i=1}^n\alpha_i S(f_i)\right\rVert ,
\]
so $T$ is an isometry.
\end{proof}

Observe that this assertion establishes the existence of a linear
isometry taking $f_i$ to $S(f_i)$; it does not establish that this
isometry agrees with $S$ on the entire unit sphere of the linear span.

\begin{proposition}\label{prop:maxfaces}
Let $X$ and $Y$ be Banach $\mathrm{SFS}$-spaces, and let
$S:\partial X_1\to\partial Y_1$ be a surjective isometry. Let $u$ be a
maximal geometric tripotent on $X$. Then:
\begin{enumerate}[label=\textup{(\roman*)}]
 \item[{\rm (i)}] $S(F_u)$ is a symmetric face of the unit ball $Y_1$;
 \item[{\rm (ii)}] the geometric tripotent corresponding to $S(F_u)$ is maximal.
\end{enumerate}
\end{proposition}

We first prove the following auxiliary lemma.

\begin{lemma}\label{lem:maximal}
Let $Z$ be an $\mathrm{SFS}$-space. The following assertions are equivalent:
\begin{enumerate}[label=\textup{(\roman*)}]
 \item[{\rm (i)}] $u$ is a maximal geometric tripotent;
 \item[{\rm (ii)}] $F_u$ is a maximal convex subset of $\partial Z_1$.
\end{enumerate}
\end{lemma}

\begin{proof}
\textup{(i)$\Rightarrow$(ii).} Since $F_u$ is a nonempty convex subset
of $\partial Z_1$, Zorn's lemma provides a maximal convex subset
$C\subseteq\partial Z_1$ containing $F_u$. By
\cite[p.~4]{Campos2021}, $C$ is a norm-exposed face of $Z_1$. Since
$Z$ is an $\mathrm{SFS}$-space, $C$ is symmetric. Thus, by
\cite[Proposition~1.6]{FR1989Affine}, there exists a unique geometric
tripotent $v$ such that $C=F_v$. Hence $F_u\subseteq F_v$. Suppose
that $F_u\neq F_v$. By \cite[Lemma~2.7]{FR1986},
\[
 Z_0(u)\cap F_v\neq\varnothing,
\]
contrary to the maximality of $u$. Consequently, $F_u=F_v=C$, and
$F_u$ is a maximal convex subset of $\partial Z_1$.

\textup{(ii)$\Rightarrow$(i).} Suppose that $Z_0(u)\neq\{0\}$. Choose
$h\in Z_0(u)$ with $\left\lVert h\right\rVert =1$. Since
$Z_2(u)\cap Z_0(u)=\{0\}$, we have $h\notin Z_2(u)$. Consider the set
\[
 C=\operatorname{co}(F_u\cup\{h\}).
\]
Since $Z_2(u)\mathbin{\diamond} Z_0(u)$, for every $f\in F_u$ and $t\in(0,1)$,
\[
 f\mathbin{\diamond}\frac{t}{1-t}h.
\]
Therefore
\[
 \left\lVert f+\frac{t}{1-t}h\right\rVert
 =\left\lVert f\right\rVert +\frac{t}{1-t}\left\lVert h\right\rVert
 =1+\frac{t}{1-t}.
\]
It follows that $\left\lVert (1-t)f+th\right\rVert =1$, so
$(1-t)f+th\in\partial Z_1$. For $t=0$ or $t=1$, this expression is
$f$ or $h$, respectively. Thus $C\subseteq\partial Z_1$. The set
$C$ is convex, $F_u\subset C$, and $F_u\neq C$. This contradicts the
maximality of $F_u$ as a convex subset of $\partial Z_1$. Hence
$Z_0(u)=\{0\}$.
\end{proof}

\begin{proof}[Proof of Proposition~\ref{prop:maxfaces}]
\textup{(i)} Since $u$ is a maximal geometric tripotent,
Lemma~\ref{lem:maximal} shows that $F_u$ is a maximal convex subset
of $\partial X_1$. By \cite[Lemma~3.5]{Tanaka2014}, $S(F_u)$ is a
maximal convex subset of $\partial Y_1$. By
\cite[p.~4]{Campos2021}, $S(F_u)$ is a norm-exposed face of $Y_1$.
Since $Y$ is an $\mathrm{SFS}$-space, $S(F_u)$ is symmetric.

\textup{(ii)} By \cite[Proposition~1.6]{FR1989Affine}, there is a
unique geometric tripotent $v$ corresponding to the symmetric face
$S(F_u)$, that is, $S(F_u)=F_v$. Since $F_v=S(F_u)$ is a maximal
convex subset of $\partial Y_1$, Lemma~\ref{lem:maximal} implies that
$v$ is maximal.
\end{proof}

The following assertion follows from this proposition.

\begin{corollary}\label{cor:maxbijection}
Let $X$ and $Y$ be Banach $\mathrm{SFS}$-spaces, and let
$S:\partial X_1\to\partial Y_1$ be a surjective isometry. Then there
exists a unique bijection
\[
 \Phi:{G\mathcal{U}}_{\max}(X)\longrightarrow{G\mathcal{U}}_{\max}(Y)
\]
determined by the equality
\[
 S(F_u)=F_{\Phi(u)}.
\]
\end{corollary}

\begin{proof}
The existence and uniqueness of $\Phi$ follow from
Proposition~\ref{prop:maxfaces}. Moreover, if $\Phi(u)=\Phi(v)$, then
\[
 S(F_u)=F_{\Phi(u)}=F_{\Phi(v)}=S(F_v).
\]
Since $S$ is injective, $F_u=F_v$. By
\cite[Proposition~1.6]{FR1989Affine}, $u=v$. Thus $\Phi$ is injective.

Now let $w\in{G\mathcal{U}}_{\max}(Y)$. Applying
Proposition~\ref{prop:maxfaces} to the surjective isometry $S^{-1}$
gives $S^{-1}(F_w)=F_u$ for some $u\in{G\mathcal{U}}_{\max}(X)$. Hence
$S(F_u)=F_w$, so $\Phi(u)=w$. Thus $\Phi$ is surjective.
\end{proof}

\begin{proposition}\label{prop:intersection}
Let $Z$ be an $\mathrm{SFS}$-space. For every $u\in{G\mathcal{U}}(Z)$, there exist
$e_1,e_2\in{G\mathcal{U}}_{\max}(Z)$ such that
\[
 F_u=F_{e_1}\cap F_{e_2}.
\]
\end{proposition}

\begin{proof}
The proof of Lemma~\ref{lem:maximal} shows that there exists
$w\in{G\mathcal{U}}_{\max}(Z)$ such that $F_u\subseteq F_w$. If $u=w$, the
proof is complete. Assume that $u\neq w$. By
\cite[Lemma~4.2]{FR1989Affine},
\[
 v=w-u\in{G\mathcal{U}}(Z),\qquad u\mathbin{\diamond} v.
\]
Thus $w=u+v$. Clearly, $u\mathbin{\diamond}(-v)$. Put
\[
 e_1=u+v=w,\qquad e_2=u-v.
\]
By \cite[Lemma~2.5]{FR1986}, $e_2$ is also a geometric tripotent.
Hence
\[
 e_1-u=v\mathbin{\diamond} u,\qquad e_2-u=-v\mathbin{\diamond} u.
\]
Applying \cite[Lemma~4.2]{FR1989Affine} again, we obtain $u\leq e_1$
and $u\leq e_2$, and therefore
\[
 F_u\subseteq F_{e_1}\cap F_{e_2}.
\]
Conversely, let $f\in F_{e_1}\cap F_{e_2}$. Then
$e_1(f)=e_2(f)=1$. Since $u=(e_1+e_2)/2$, we have
\[
 u(f)=\frac{e_1(f)+e_2(f)}2=1.
\]
Thus $f\in F_u$, and hence $F_{e_1}\cap F_{e_2}\subseteq F_u$.
Combining the two inclusions gives
\[
 F_u=F_{e_1}\cap F_{e_2}.
\]

It remains to show that $e_2$ is maximal. By
\cite[Lemma~1.8]{FR1989Affine},
\[
 Z_0(e_1)=Z_0(u+v)=Z_0(u)\cap Z_0(v).
\]
Since $e_1=w$ is maximal,
$\{0\}=Z_0(e_1)=Z_0(u)\cap Z_0(v)$. As $Z_0(-v)=Z_0(v)$, we obtain
\[
 \{0\}=Z_0(u)\cap Z_0(-v)=Z_0(u-v).
\]
Thus $e_2$ is maximal.
\end{proof}

\begin{theorem}\label{thm:faces}
Let $X$ and $Y$ be Banach $\mathrm{SFS}$-spaces, and let
$S:\partial X_1\to\partial Y_1$ be a surjective isometry. Then, for
every symmetric face $F$ of the unit ball $X_1$, its image $S(F)$ is a
symmetric face of the unit ball $Y_1$.
\end{theorem}

\begin{proof}
Let $F=F_u$, where $u\in{G\mathcal{U}}(X)$. By
Proposition~\ref{prop:intersection}, there exist
$e_1,e_2\in{G\mathcal{U}}_{\max}(X)$ such that
\[
 F_u=F_{e_1}\cap F_{e_2}.
\]
Put $v_i=\Phi(e_i)$, $i=1,2$, where $\Phi$ is the bijection from
Corollary~\ref{cor:maxbijection}. Then
\[
 S(F_{e_i})=F_{v_i},\qquad i=1,2.
\]
Since $S$ is injective,
\[
 S(F_u)=S(F_{e_1}\cap F_{e_2})
 =S(F_{e_1})\cap S(F_{e_2})=F_{v_1}\cap F_{v_2}.
\]
Consider the functional
\[
 w=\frac{v_1+v_2}2\in Y^*.
\]
Clearly, $\left\lVert w\right\rVert \leq 1$. Since $F_u\neq\varnothing$, choose
$y_0\in S(F_u)$. Then $v_1(y_0)=v_2(y_0)=1$, so $w(y_0)=1$.
Since $\left\lVert y_0\right\rVert =1$, we obtain $\left\lVert w\right\rVert =1$.

We now show that $F_w=F_{v_1}\cap F_{v_2}$. The inclusion
$F_{v_1}\cap F_{v_2}\subseteq F_w$ follows directly from the
definition of $w$. Conversely, if $y\in F_w$, then
\[
 1=\operatorname{Re}w(y)
 =\frac{\operatorname{Re}v_1(y)+\operatorname{Re}v_2(y)}2.
\]
Since $|v_i(y)|\leq 1$, we have $\operatorname{Re}v_i(y)\leq 1$.
Hence $\operatorname{Re}v_1(y)=\operatorname{Re}v_2(y)=1$, and
therefore $v_1(y)=v_2(y)=1$. Thus $y\in F_{v_1}\cap F_{v_2}$.
Consequently,
\[
 S(F_u)=F_w.
\]
This is a nonempty norm-exposed face of the unit ball $Y_1$. Since
$Y$ is an $\mathrm{SFS}$-space, it is symmetric.
\end{proof}

\begin{theorem}\label{thm:tripotents}
Let $X$ and $Y$ be Banach $\mathrm{SFS}$-spaces, and let
$S:\partial X_1\to\partial Y_1$ be a surjective isometry. Then there
exists a unique bijection
\[
 \Delta:{G\mathcal{U}}(X)\longrightarrow{G\mathcal{U}}(Y),
 \qquad \left.\Delta\right|_{{G\mathcal{U}}_{\max}(X)}=\Phi,
\]
determined by the equality
\[
 S(F_u)=F_{\Delta(u)}.
\]
It preserves order in both directions:
\[
 u\leq v\quad\Longleftrightarrow\quad\Delta(u)\leq\Delta(v).
\]
\end{theorem}

\begin{proof}
By Theorem~\ref{thm:faces}, for every $u\in{G\mathcal{U}}(X)$, the set $S(F_u)$
is a symmetric face of $Y_1$. Thus, by
\cite[Proposition~1.6]{FR1989Affine}, there exists a unique
$v\in{G\mathcal{U}}(Y)$ such that $S(F_u)=F_v$. Putting $\Delta(u)=v$, we
define the map
\[
 \Delta:{G\mathcal{U}}(X)\longrightarrow{G\mathcal{U}}(Y),\qquad S(F_u)=F_{\Delta(u)}.
\]
If $\Delta(u_1)=\Delta(u_2)$, then
\[
 S(F_{u_1})=F_{\Delta(u_1)}=F_{\Delta(u_2)}=S(F_{u_2}).
\]
Since $S$ is injective, $F_{u_1}=F_{u_2}$. By
\cite[Proposition~1.6]{FR1989Affine}, $u_1=u_2$. Thus $\Delta$ is
injective.

Now let $v\in{G\mathcal{U}}(Y)$. Applying Theorem~\ref{thm:faces} to the
surjective isometry $S^{-1}$, we conclude that $S^{-1}(F_v)$ is a
symmetric face of $X_1$. Hence $S^{-1}(F_v)=F_u$ for some
$u\in{G\mathcal{U}}(X)$. Consequently,
\[
 F_v=S(F_u)=F_{\Delta(u)},
\]
so $v=\Delta(u)$. Thus $\Delta$ is surjective.

Finally,
\begin{align*}
 u\leq v
 &\Longleftrightarrow F_u\subseteq F_v
 \Longleftrightarrow S(F_u)\subseteq S(F_v)\\
 &\Longleftrightarrow F_{\Delta(u)}\subseteq F_{\Delta(v)}
 \Longleftrightarrow\Delta(u)\leq\Delta(v).
\end{align*}
\end{proof}

Under the assumptions and notation of Theorem~\ref{thm:tripotents},
the following assertions hold.

\begin{corollary}\label{cor:invariants}

{\rm (i)} For every $u\in{G\mathcal{U}}(X)$, we have
 \[
  S(F_{-u})=-S(F_u),\qquad\Delta(-u)=-\Delta(u).
 \]
\begin{enumerate}
    \item[{\rm (ii)}] If $x$ is a norm-exposed point of the unit ball $X_1$, then
 \[
  S(-x)=-S(x).
 \]

\item[{\rm (iii})] If $X$ has property $\mathrm{PE}$, then $S(-x)=-S(x)$ for every
 $x\in\operatorname{ext} X_1$.
 
 \item[{\rm (iv)}] For every $u\in{G\mathcal{U}}(X)$, the restriction
 \[
  \left.S\right|_{\partial(X_0(u))_1}:
  \partial(X_0(u))_1\longrightarrow\partial(Y_0(\Delta(u)))_1
 \]
 is a surjective isometry.
 
\item[{\rm (v)}] For all $u,v\in{G\mathcal{U}}(X)$, we have
 \[
  u\mathbin{\diamond} v\quad\Longleftrightarrow\quad\Delta(u)\mathbin{\diamond}\Delta(v),
 \]
 and
 \[
  u\mathbin{\diamond} v\quad\Longrightarrow\quad
  \Delta(u+v)=\Delta(u)+\Delta(v).
 \]
 
 \item[{\rm (vi)}] For every $w\in{G\mathcal{U}}(X)$, the map
 \[
  \Delta:L_w\longrightarrow L_{\Delta(w)}
 \]
 preserves order and orthocomplementation, where
 \[
  L_w=\{u\in{G\mathcal{U}}:u\leq w\}\cup\{0\},\qquad\Delta(0)=0.
 \]
\end{enumerate}
\end{corollary}

\begin{proof}
\textup{(i)} Let $u\in{G\mathcal{U}}(X)$. By
Proposition~\ref{prop:intersection}, there exist
$e_1,e_2\in{G\mathcal{U}}_{\max}(X)$ such that
\[
 F_u=F_{e_1}\cap F_{e_2}.
\]
By Lemma~\ref{lem:maximal}, $F_{e_1}$ and $F_{e_2}$ are maximal
convex subsets of $\partial X_1$. Hence
\cite[Proposition~2.3]{Mori2018} gives
\[
 S(F_{-e_i})=-S(F_{e_i}),\qquad i=1,2.
\]
Using $F_{-u}=-F_u$ and the injectivity of $S$, we obtain
\begin{align*}
 S(F_{-u})
 &=S\bigl((-F_{e_1})\cap(-F_{e_2})\bigr)\\
 &=S(-F_{e_1})\cap S(-F_{e_2})\\
 &=(-S(F_{e_1}))\cap(-S(F_{e_2}))\\
 &=-\bigl(S(F_{e_1})\cap S(F_{e_2})\bigr)=-S(F_u).
\end{align*}
By the definition of $\Delta$,
\[
 F_{\Delta(-u)}=S(F_{-u})=-S(F_u)=-F_{\Delta(u)}=F_{-\Delta(u)}.
\]
The uniqueness of the geometric tripotent associated with a symmetric
face (see \cite[Proposition~1.6]{FR1989Affine}) implies
$\Delta(-u)=-\Delta(u)$.

\textup{(ii)} Since $x$ is norm exposed and $X$ is an $\mathrm{SFS}$-space,
the singleton face $\{x\}$ is symmetric. Hence there exists
$u\in{G\mathcal{U}}(X)$ such that $\{x\}=F_u$. By \textup{(i)},
\[
 \{S(-x)\}=S(F_{-u})=-S(F_u)=\{-S(x)\}.
\]
Therefore $S(-x)=-S(x)$.

\textup{(iii)} Under property $\mathrm{PE}$, every $x\in\operatorname{ext} X_1$ is norm
exposed, so \textup{(ii)} gives $S(-x)=-S(x)$.

\textup{(iv)} By Proposition~\ref{prop:orthogonality}\textup{(ii)},
\begin{align*}
 S(\partial X_1\cap X_0(u))
 &=S(\partial X_1\cap F_u^\diamond)\\
 &=\partial Y_1\cap S(F_u)^\diamond\\
 &=\partial Y_1\cap F_{\Delta(u)}^\diamond
 =\partial Y_1\cap Y_0(\Delta(u)).
\end{align*}

\textup{(v)} By \cite[Lemma~2.5]{FR1986}, $u\mathbin{\diamond} v$ if and only if
$F_u\mathbin{\diamond} F_v$. Thus
\begin{align*}
 u\mathbin{\diamond} v
 &\Longleftrightarrow F_u\mathbin{\diamond} F_v
 \Longleftrightarrow S(F_u)\mathbin{\diamond} S(F_v)\\
 &\Longleftrightarrow F_{\Delta(u)}\mathbin{\diamond} F_{\Delta(v)}
 \Longleftrightarrow\Delta(u)\mathbin{\diamond}\Delta(v).
\end{align*}
Since $\Delta$ is an order isomorphism and, by
\cite[Lemma~3.2(b)(iii)]{Ibragimov2024}, $u\mathbin{\diamond} v$ implies
$u+v=u\vee v$, we have
\[
 \Delta(u+v)=\Delta(u\vee v)
 =\Delta(u)\vee\Delta(v)=\Delta(u)+\Delta(v)
 \qquad(u\mathbin{\diamond} v).
\]

\textup{(vi)} Order preservation was established in
Theorem~\ref{thm:tripotents}. To prove preservation of
orthocomplementation, let $u\in L_w$. Then $u\mathbin{\diamond}(w-u)$, and hence
\[
 \Delta(w)=\Delta(w-u+u)=\Delta(w-u)+\Delta(u).
\]
Thus $\Delta(w-u)=\Delta(w)-\Delta(u)$.
\end{proof}

In \cite{Wang2002}, Tingley's problem was solved for $\ell_1$-sums of
strictly convex normed spaces by describing surjective isometries between
their unit spheres. Every real Hilbert space is a strictly convex
$\mathrm{SFS}$-space. Thus, by \cite[Theorem~3.1]{Ibragimov2018}, the spaces
\[
 X=\left(\bigoplus_{i\in I}H_i\right)_{\ell_1},
 \qquad Y=\left(\bigoplus_{j\in J}K_j\right)_{\ell_1},
\]
where $H_i$ and $K_j$ are nonzero real Hilbert spaces, are also
$\mathrm{SFS}$-spaces. Consequently, the result of \cite{Wang2002} gives a
positive solution to Tingley's problem for this class of $\mathrm{SFS}$-spaces.

The following structural theorem was established in
\cite{Kudaybergenov2020}.

\begin{theorem}[{\cite[Theorem~3.1]{Kudaybergenov2020}}]\label{thm:classification}
Let $Z$ be a finite-dimensional real neutral $\mathrm{SFS}$-space with property
$\mathrm{JP}$. Then
\[
 Z\cong\bigoplus_{i=1}^k{}_{\ell_1} H_i,
\]
where $H_i$, $i=1,\ldots,k$, are nonzero finite-dimensional real Hilbert
spaces, $k=\operatorname{rank} Z$, and the isomorphism is linear and isometric.
\end{theorem}

This theorem and the result of \cite{Wang2002} yield the following
assertion.

\begin{corollary}\label{cor:finiteextension}
Let $X$ and $Y$ be finite-dimensional real neutral $\mathrm{SFS}$-spaces with
property $\mathrm{JP}$. Then every surjective isometry
\[
 S:\partial X_1\longrightarrow\partial Y_1
\]
extends uniquely to a surjective linear isometry $T:X\to Y$.
\end{corollary}

\begin{theorem}\label{thm:extension}
Let $X$ and $Y$ be real Banach neutral $\mathrm{SFS}$-spaces with property
$\mathrm{JP}$. Suppose that there exists a family
$\{u_\alpha\}_{\alpha\in I}\subset{G\mathcal{U}}(X)$ satisfying the following
conditions:
\begin{enumerate}[label=\textup{(\arabic*)}]
 \item[{\rm(1)}] $X_0(u_\alpha)$ is finite dimensional for every $\alpha\in I$;
 \item[{\rm(2)}] for every $\alpha,\beta\in I$, there exists $\gamma\in I$ such
 that
 \[
  X_0(u_\alpha)\cup X_0(u_\beta)\subseteq X_0(u_\gamma);
 \]
 \item[{\rm(3)}]
 \[
  \overline{\bigcup_{\alpha\in I}X_0(u_\alpha)}=X,
 \]
 where the closure is taken in norm.
\end{enumerate}
Then every surjective isometry $S:\partial X_1\to\partial Y_1$
extends uniquely to a surjective linear isometry $T:X\to Y$.
\end{theorem}

\begin{proof}
By Corollary~\ref{cor:invariants}\textup{(iv)}, the maps
\[
 S_\alpha:\partial(X_0(u_\alpha))_1
 \longrightarrow\partial(Y_0(\Delta(u_\alpha)))_1,
 \qquad S_\alpha:=\left.S\right|_{\partial(X_0(u_\alpha))_1},
\]
are surjective isometries. Thus the sphere
$\partial(Y_0(\Delta(u_\alpha)))_1$ is compact, and therefore
$Y_0(\Delta(u_\alpha))$ is finite dimensional. By
\cite[Theorem~3.6 and Proposition~4.1]{FR1989Affine}, the spaces
$X_0(u_\alpha)$ and $Y_0(\Delta(u_\alpha))$ are neutral
$\mathrm{SFS}$-spaces. The fact that these subspaces inherit property $\mathrm{JP}$
follows from \cite[Lemma~3.5]{Kudaybergenov2020}. Hence
Corollary~\ref{cor:finiteextension} provides a linear isometric
extension
\[
 T_\alpha:X_0(u_\alpha)\longrightarrow Y_0(\Delta(u_\alpha)).
\]
For $x\in X_0(u_\alpha)\setminus\{0\}$,
\begin{equation}\tag{3.5}\label{eq:radial}
 T_\alpha(x)
 =\left\lVert x\right\rVert \,T_\alpha\left(\frac{x}{\left\lVert x\right\rVert }\right)
 =\left\lVert x\right\rVert \,S_\alpha\left(\frac{x}{\left\lVert x\right\rVert }\right)
 =\left\lVert x\right\rVert \,S\left(\frac{x}{\left\lVert x\right\rVert }\right).
\end{equation}
Condition \textup{(2)} implies that the set
\[
 E=\bigcup_{\alpha\in I}X_0(u_\alpha)
\]
is a subspace. Indeed, let $x_1,x_2\in E$. Then
$x_1\in X_0(u_\alpha)$ and $x_2\in X_0(u_\beta)$ for some
$\alpha,\beta\in I$. Since
$X_0(u_\alpha)\cup X_0(u_\beta)\subseteq X_0(u_\gamma)$ for some
$\gamma\in I$, and $X_0(u_\gamma)$ is a linear subspace,
\[
 ax_1+bx_2\in X_0(u_\gamma)\subseteq E
 \qquad(a,b\in\mathbb{R}).
\]

Let $T$ denote the map defined by the radial formula
\eqref{eq:radial}:
\[
 T(0)=0,\qquad T(x)=\left\lVert x\right\rVert \,S\left(\frac{x}{\left\lVert x\right\rVert }\right)
 \quad(x\in E\setminus\{0\}).
\]
Clearly, $\left.T\right|_{X_0(u_\alpha)}=T_\alpha$, where $T_\alpha$
is a linear isometry. Let $x_1,x_2\in E$ and $a,b\in\mathbb{R}$. Choose
$\gamma\in I$ such that $x_1,x_2\in X_0(u_\gamma)$. Then
$ax_1+bx_2\in X_0(u_\gamma)\subseteq E$, and
\begin{align*}
 T(ax_1+bx_2)
 &=T_\gamma(ax_1+bx_2)\\
 &=aT_\gamma(x_1)+bT_\gamma(x_2)
 =aT(x_1)+bT(x_2).
\end{align*}
Moreover,
\[
 \left\lVert T(x_1)-T(x_2)\right\rVert
 =\left\lVert T_\gamma(x_1)-T_\gamma(x_2)\right\rVert =\left\lVert x_1-x_2\right\rVert .
\]
Thus $T:E\to Y$ is a linear isometry.

Since $E$ is dense in $X$ and $Y$ is complete, $T:E\to Y$ extends
uniquely to a linear isometry $T:X\to Y$. Let $x\in\partial X_1$.
Choose nonzero $z_n\in E$ such that $z_n\to x$, and put
\[
 x_n=\frac{z_n}{\left\lVert z_n\right\rVert }.
\]
Then $x_n\to x$ and $T(x_n)=S(x_n)$. By continuity, $T(x)=S(x)$.

For every nonzero $y\in Y$, there exists $x\in\partial X_1$ such
that
\[
 S(x)=\frac{y}{\left\lVert y\right\rVert }.
\]
Hence $T(\left\lVert y\right\rVert x)=y$, so $T$ is surjective. Finally, every linear
extension of $S$ must satisfy
\[
 T(x)=\left\lVert x\right\rVert \,S\left(\frac{x}{\left\lVert x\right\rVert }\right)
 \qquad(x\neq 0),
\]
which proves uniqueness.
\end{proof}

%%%%%%%%%%%%%%%%%%%%%%%%%%%%%%%%%%%%%%%%%%%%%%%%%%%%%%%%%%
\vspace{1cm}
\textbf{Author contributions.}
All authors contributed to the conception and development of the study, participated in the proofs and preparation of the manuscript, and read and approved the final manuscript.

\textbf{Funding.}
The authors received no funding for this work.

\textbf{Data availability.}
No datasets were generated or analysed during the current study.
%%%%%%%%%%%%%%%%%%%%%%%%%%%%%%%%%%%%%%%%%%%%%%%%%%%%%%%%%%%%%

\section*{Declarations}
\,

\textbf{Conflict of interest.}
The authors declare that they have no competing interests.
%                           %%
%%===========================================================================================%%

\end{document}